\documentclass[12pt]{amsart}
\usepackage{amsmath,amsthm,amssymb,mathrsfs,latexsym,amsfonts}
\usepackage{amssymb}
\usepackage{bbding}
\usepackage[T1]{fontenc}
\usepackage[utf8]{inputenc}
\usepackage{lscape}
\usepackage{tabu}
\usepackage{a4wide}

\usepackage[dvipsnames]{xcolor}

\newtheorem{Main}{Theorem}
\newtheorem{Coro}{Corollary}

\newtheorem{theorem}{Theorem}[section]
\newtheorem{lemma}[theorem]{Lemma}
\newtheorem{proposition}[theorem]{Proposition}
\theoremstyle{definition}

\theoremstyle{remark}

\numberwithin{equation}{section}

\begin{document}
\bibliographystyle{plain}

\title[]{Reducibility of ultra-differentiable quasiperiodic cocycles under an adapted arithmetic condition}

\author{Abed Bounemoura}
%%    Address of record for the research reported here
\address{CNRS - PSL Research University\\
    (Universit{\'e} Paris-Dauphine and Observatoire de Paris)}
%%    Current address
%\curraddr{Department of Mathematics and Statistics,
%Case Western Reserve University, Cleveland, Ohio 43403}
\email{abedbou@gmail.com}
\author{Claire Chavaudret}
\address{Institut de Math\'ematiques de Jussieu, Universit\'e Paris Diderot}
\email{chavaudr@math.univ-paris-diderot.fr}

\author{Shuqing Liang}
\address{School of Mathematics, Jilin University, 130012 Changchun, P. R. China;}
\address{CNRS - PSL Research University\\     (CEREMADE, Universit{\'e} Paris-Dauphine)}
\email{liangshuqing@jlu.edu.cn; liang@ceremade.dauphine.fr}

\thanks{The first two authors were supported by ANR BeKAM; second author is grateful to G.Popov (Univ. Nantes) for useful discussions. S. Liang  gratefully acknowledges financial support from China Scholarship Council, and is supported by National Natural Science Foundation of China Grant No.11501240 and No.11671071.}

%    General info

%\date{2019}



\begin{abstract}
We prove a reducibility result for $sl(2,\mathbb{R})$ quasi-periodic cocycles close to a constant elliptic matrix in ultra-differentiable classes, under an adapted arithmetic condition which extends the Brjuno-R\"{u}ssmann condition in the analytic case. The proof is based on an elementary property of the fibered rotation number and deals with ultra-differentiable functions with a weighted Fourier norm. We also show that a weaker arithmetic condition is necessary for reducibility, and that it can be compared to a sufficient arithmetic condition.
\end{abstract}

\maketitle

\section{Introduction}

%\cite{eliasson1997discrete}

We will study the following time-quasiperiodic linear system, or quasiperiodic cocycle
\begin{eqnarray*}\label{eq:first}
\left\{
\begin{array}{ll}
 x'(t) = (A + F(\theta(t)))x(t),\\
\theta'(t)=\omega,
 \end{array}
\right.
\end{eqnarray*}
where $x\in\mathbb{R}^2$, $\theta\in\mathbb{T}^d=\mathbb{R}^d/ \mathbb{Z}^d$ with an integer $d \geq 1$, $\omega\in\mathbb{R}^d$ is a non-resonant frequency vector, $A$ is elliptic (meaning it is conjugated to some non-zero element in $so(2,\mathbb{R})$) and $F : \mathbb{T}^d \rightarrow sl(2,\mathbb{R})$ belongs to some ultra-differentiable class. Let us recall that $sl(2,\mathbb{R})$ is the Lie algebra of traceless matrices and $so(2,\mathbb{R})$ is the Lie sub-algebra of skew-symmetric matrices. Such a quasi-periodic cocycle is said to be reducible if the time-quasiperiodic linear system can be conjugated, by a time quasi-periodic transformation, to a constant (time-independent) linear system.  

One of the main motivation for studying reducibility of quasi-periodic cocycles came from the Schr\"{o}dinger equation
\[
-y''(t)+q(\theta+\omega t) y(t) = E y(t)
\]
and the question of existence of so-called Floquet solutions (which always exist when $d=1$); this is readily seen to be equivalent to the reducibility of a family of quasi-periodic cocycle depending on the "energy" parameter $E$. In a pioneering work, Dinaburg and Sinai \cite{dinaburg1975one} proved that for a small analytic potential $q$ (or large energy $E$), for a set of positive (and asymptotically full) measure of energy $E$ in the spectrum, the associated cocycle is reducible provided the frequency $\omega\in\mathbb{R}^d$ is Diophantine:
\[
|k\cdot \omega|\geq\frac{\gamma}{|k|^\tau},\quad k\in \mathbb{Z}^d\backslash\{0\},
\]
for some constant $ \gamma>0$ and $\tau\geq d-1$, where $|k|$ is the sum of the absolute values of the components and $k\cdot \omega$ the Euclidean inner product. For a fixed cocycle as we considered above, their result amounts to a reducibility result under a Diophantine condition (with respect to $\omega$) on the so-called fibered rotation number $\rho=\rho(A+F)$, for a small analytic $F$, the smallness assumption depending on $\rho$. R\"{u}ssmann \cite{russmann1980one} later extended this result, under a more general arithmetic condition on $\omega$ and $\rho$ (this condition, weaker than the Diophantine condition, is slightly stronger than the so-called the Brjuno-R\"{u}ssmann condition). Moser and P\"{o}schel \cite{moser1984extension} further extended the result to include some rational fibered rotation numbers, using a technique of resonance-cancellation, but the breakthrough came from Eliasson~ \cite{eliasson1992floquet}: by sharpening this resonance-cancellation technique he obtained the reducibility for a set of full measure of fibered rotation number.  

Since then, many works have been devoted to the reducibililty of quasi-periodic cocycles, in different regularity classes and for cocycles taking values in different Lie algebras. In particular, many non-perturbative results have been obtained but they are restricted to two-dimensional frequencies $\omega \in \mathbb{R}^2$; we should not try to review to state of the art as we will be interested only in perturbative results, but valid in any dimension $\omega \in \mathbb{R}^d$. More precisely, we will be interested in the interaction between the regularity of the cocycle and the arithmetic properties of the frequency vector $\omega$.

For smooth cocycles, the Diophantine condition on $\omega$ is known to be sufficient, and it is not hard to see it is also necessary. The analytic case is more subtle. Chavaudret and Marmi \cite{chavaudret2012reducibility} extended the result of R\"{u}ssmann to obtain reducibility under the Brjuno-R\"{u}ssmann condition; this sufficient arithmetic condition is not know to be optimal, but can be compared to a natural necessary condition (that we call the R\"{u}ssmann condition). The proof in~\cite{chavaudret2012reducibility}  uses ideas of R\"{u}ssmann \cite{russmann2010kam} and P\"{o}schel \cite{poschel2011kam}, which deal with the corresponding results for respectively Hamiltonian systems and vector fields on the torus. Those results were later extended by Bounemoura and F\'{e}joz \cite{bounebouraACMSP}, \cite{BounemouraMAMS} for a more general class of systems with ultra-differentiable regularity. A particular case is the $\alpha$-Gevrey regularity, for a real parameter $\alpha \geq 1$, for which the analytic case is recovered by setting $\alpha=1$; an $\alpha$-Brjuno-R\"{u}ssmann condition is introduced in~\cite{bounebouraACMSP}  (an equivalent condition was independently obtained in~\cite{portugais} for vector fields on the torus) and exactly as for $\alpha=1$, this sufficient condition is shown to be comparable to the natural necessary condition. Unfortunately the results in~\cite{BounemouraMAMS} do not allow such a comparison in general as the sufficient arithmetic condition is affected by a technical assumption, and so the results are not as accurate as those obtained in the Gevrey case in~\cite{bounebouraACMSP}. 

The purpose of this article is to improve the results of \cite{BounemouraMAMS} within the context of quasi-periodic cocycles: we will obtain a result of reducibility valid for a larger class of ultra-differentiable systems, with a better sufficient arithmetic condition in the sense that it can be compared to the natural necessary condition.

\section{Statement of the main results}

Let us recall the setting. We have $\mathbb{T}^d=\mathbb{R}^d/\mathbb{Z}^d$ for an integer $d\geq1$ and we consider the cocycle
\begin{eqnarray}\label{eq:aim}
\left\{
\begin{array}{ll}
 x'(t) = (A + F(\theta(t)))x(t),\\
\theta'(t)=\omega,
 \end{array}
\right.
\end{eqnarray}
where $x\in\mathbb{R}^2$, $\theta\in\mathbb{T}^d$ and $A\in sl(2,\mathbb{R})$ is an elliptic matrix, or equivalently, fixing $\theta(0)=0 \in\mathbb{T}^d$, we consider
\[  x'(t) = (A + F(t\omega))x(t).\] 
Such a cocycle will be simply denoted by $(\omega,A+F)$. It is said to be reducible if there exists a (fibered) conjugacy between $(A+F)$ and a constant cocycle: there exist $Y : \mathbb{T}^d \rightarrow GL(2,\mathbb{R})$ and an elliptic matrix $B \in sl(2,\mathbb{R})$ such that
\[
\partial_\omega Y = (A+F)Y-YB,
\]
where $\partial_\omega Y(\theta)=\partial_\theta Y(\theta)\cdot \omega$. Since the matrices actually take values in $sl(2,\mathbb{R})$, a perhaps more natural definition would be to require that the conjugacy $Y$ takes value in the corresponding Lie group $SL(2,\mathbb{R})$: it follows from the work in~\cite{chavaudret2011reducibility} that these two definitions are the same. In order to get such a reducibility, one need to impose regularity assumptions on $F$ and an arithmetic condition on $\omega$ (we will also impose a similar arithmetic condition on the fibered rotation number).
 
To quantify the regularity of $F : \mathbb{T}^d \rightarrow sl(2,\mathbb{R})$, we introduce a weight function
\[ \Lambda : [1,+\infty)\to[1,+\infty) \]
which we assume is increasing and differentiable. Expanding a smooth function $f\in C^{\infty}(\mathbb{T}^d,\mathbb{R})$ in Fourier series
%\begin{eqnarray*}%\label{eq:matrix fourier expansion}
\[ f(\theta) = \sum_k \hat{f}(k)e^{2\pi i k\cdot\theta}\]
%\end{eqnarray*}
we will say it is $\Lambda$-ultra-differentiable if there exists $r>0$ such that
\begin{equation}\label{eq:norm of ultradifferentiable}
|f|_{r}=|f|_{\Lambda,r}:=\sum_{k\in\mathbb{Z}^d} |\hat{f}(k)|e^{2\pi \Lambda(|k|)r}<\infty.
\end{equation}
The ultra-differentiable weighted norm of $f$ is then defined by \eqref{eq:norm of ultradifferentiable}, and we say $f$ belongs to the ultra-differentiable function class $U_{\Lambda,r}(\mathbb{T}^d,\mathbb{R})$. This defines a Banach space. We will require the function $\Lambda$ to be subadditive, namely
\begin{equation}\label{S}
\Lambda(x+y)\leq \Lambda(x)+\Lambda(y), \quad x,y\geq 1.\tag{S}
\end{equation}
This assumption turns $U_{\Lambda,r}(\mathbb{T}^d,\mathbb{R})$ into a Banach algebra (for a proof of this elementary fact, see Appendix~\ref{app1}). Now for a matrix-valued function $M:\mathbb{T}^d\to M_2(\mathbb{R})$, one extends the definition of~\eqref{eq:norm of ultradifferentiable} in such a way that $U_{\Lambda,r}(\mathbb{T}^d,M_2(\mathbb{R}))$ becomes a Banach algebra for the product of matrices. In the sequel, we will use the notation $U_\Lambda=\bigcup_{r >0} U_{\Lambda,r}$ when convenient. Main examples of ultra-differentiable classes are the $\alpha$-Gevrey class associated to $\Lambda_\alpha(v) = v^{\frac{1}{\alpha}}$ for $\alpha\geq1$ and the real analytic class for $\Lambda_1(v) = v$, but many more examples are readily available. In particular, the quasi-analytic class, i.e the class of functions which are uniquely determined by the sequence of their derivatives at a point, corresponds to a function $\Lambda$ satisfying
\begin{equation}\label{quasi-an}
\int_1^\infty \frac{\Lambda(v)}{v^2} dv =+\infty
\end{equation}
The parameter $r$ can be called "ultra-differentiable parameter"  if \eqref{eq:norm of ultradifferentiable} holds, and it is essentially the radius of convergence for real-analytic functions.

Next we need to quantify the non-resonance condition on  $\omega\in\mathbb{R}^d$. To do so, we introduce an approximating function
\[
 \Psi: [1,+\infty)\to[1,+\infty) 
 \]
which we assume, without loss of generality, to be increasing and differentiable, and for which
\begin{equation}\label{eq:BR frequency omega}
\Psi(K)=\max\{|2\pi k\cdot \omega|^{-1} \; | \; 0 < |k| \leq K\}, \quad K \in \mathbb{N}.
 \end{equation}
We also need to quantify the non-resonance condition on the fibered rotation number $\rho= \rho(A + F)$, a definition of which is recalled in Appendix~\ref{app2}. Without loss of generality, we use the same approximating function and requires that
\begin{equation}\label{eq:BR frequency rho}
|2\rho\pm2\pi  k\cdot \omega | \geq \frac{1}{\Psi(K)},\quad 0 < |k| \leq K.
\end{equation}
The approximating function $\Psi$ will be assumed to satisfy the following arithmetic condition adapted to the weight $\Lambda$, that we call the $\Lambda$-Brjuno-R\"{u}ssmann condition
\begin{equation}\label{Assumption H3-1:convergenceofRadius}
\int_1^{+\infty} \frac{\Lambda'(v)\ln\Psi(v)dv}{\Lambda^2(v)} <\infty.
 \tag{$\Lambda$-BR}
\end{equation} 
One easily check that the last condition is equivalent to
\begin{equation}\label{BR}
\int_1^{+\infty} \frac{\Psi'(v)dv}{\Psi(v)\Lambda(v)} <\infty.
\end{equation} 
In the Gevrey case $\Lambda_\alpha(v) = v^{\frac{1}{\alpha}}$ (and thus in the analytic case when $\alpha=1$), the $\Lambda_\alpha$-Brjuno-R\"{u}ssmann condition is
\begin{equation}\label{Gev}
\int_1^{+\infty} \frac{\ln\Psi(v)dv}{v^{1+\frac{1}{\alpha}}}
\end{equation} 
and one recovers the $\alpha$-Brjuno-R\"{u}ssmann condition introduced in~\cite{bounebouraACMSP} (for $\alpha=1$, this is the Brjuno-R\" ussmann condition as in \cite{chavaudret2012reducibility}). 

\bigskip
\noindent Now we can state the main theorem of this paper.

\begin{Main}\label{th: main theorem}
Assume that $\Lambda$ satisfy~\eqref{S} and $\omega$ and $\rho$ verify~\eqref{eq:BR frequency omega} and~\eqref{eq:BR frequency rho} with $\Psi$ satisfying the $\Lambda$-Brjuno-R\"{u}ssmann condition~\eqref{Assumption H3-1:convergenceofRadius}. Given any $r>0$ and any quasiperiodic cocycle $(\omega,A+F)$ as in~\eqref{eq:aim}, with a non-zero elliptic matrix $A\in sl(2,\mathbb{R})$ and $F\in U_{\Lambda,r}(\mathbb{T}^d,sl(2,\mathbb{R}))$, there exists $\bar{\varepsilon}$ depending only on $r,\Lambda,A,\omega$ such that if $|F|_r\leq \bar{\varepsilon}$,  the cocycle $(\omega,A+F)$ is reducible with a conjugacy $Y \in U_{\Lambda,{r}/{2}}(\mathbb{T}^d, GL(2,\mathbb{R}))$ which satisfies $|Y|_{\frac{r}{2}}\leq 2$ and $|Y^{-1}|_{\frac{r}{2}}\leq 2$.
\end{Main}

The $\Lambda$-Brjuno-R\"{u}ssmann condition~\eqref{Assumption H3-1:convergenceofRadius} is thus sufficient for the reducibility within the class $U_\Lambda$; we do not know if the condition is necessary yet it implies the following $\Lambda$-R\"{u}ssmann condition
\begin{equation}\label{Rus}
\lim_{v\to\infty}\frac{\log\Psi(v)}{\Lambda(v)}=0\tag{$\Lambda$-R}
\end{equation}
which is necessary as the next statement shows. 

\begin{Main}\label{th2}
Assume that $\Lambda$ satisfy~\eqref{S}, $\omega$ verify~\eqref{eq:BR frequency omega} with $\Psi$ not satisfying the $\Lambda$-R\"{u}ssmann condition~\eqref{Rus} and $\rho$ is arbitrary. Then there exists $r>0$ such that for all $\varepsilon>0$, there exist a quasiperiodic cocycle $(\omega,A+F)$ as in~\eqref{eq:aim}, with a non-zero elliptic matrix $A\in sl(2,\mathbb{R})$ and $F\in U_{\Lambda,r}(\mathbb{T}^d,sl(2,\mathbb{R}))$ satisfying $|F|_r \leq \varepsilon$ which is not reducible by any continuous conjugacy $Y : \mathbb{T}^d \rightarrow GL(2,\mathbb{R})$.

\end{Main}

%ajout claire
\section{Corollaries and applications}
\subsection{A special case: quasi-analytic functions}\label{quasi-analytic}

In the $\alpha$-Gevrey case when $\Lambda_\alpha(v) = v^{\frac{1}{\alpha}}$, in view of~\eqref{Gev} the condition~\eqref{Assumption H3-1:convergenceofRadius} holds true for the approximating function $\Psi(v) = e^{v^\beta}$ for any $\beta<1/\alpha$; in particular in the analytic case when $\alpha=1$ it holds true for any $\beta<1$.

As a matter of fact, the latter also holds true for quasi-analytic functions, that is if $\Lambda$ satisfies~\eqref{quasi-an}, then the condition~\eqref{Assumption H3-1:convergenceofRadius} holds true for $\Psi(v) = e^{v^\beta}$ for any $\beta<1$. Indeed, there exists 
\[ \frac{1+\beta}{2}<\gamma <1, \quad 0< \delta < \frac{1-\beta}{2}\] 
and an unbounded sequence $(v_n)$ such that 
$1\leq v_{n+1}-v_n \leq n^\delta$ and $\Lambda(v_n)>v_n^\gamma$. In particular $v_n\geq n $ for all $n$. Thus for all $n\geq 1$, recalling that~\eqref{Assumption H3-1:convergenceofRadius} is equivalent to~\eqref{BR}, one has
\begin{equation*}
\int_1^{v_{n+1}}\frac{\Psi'(v)dv}{\Psi(v)\Lambda(v)} \leq \beta\sum_{k\leq n} \frac{v_k^{\beta-1}}{\Lambda(v_k)}(v_{k+1}-v_k) \leq \beta\sum_{k\leq n}\frac{1}{k^{1+\gamma - \beta -\delta}}
\end{equation*}
and this last sum converges as $n\rightarrow +\infty$.

\subsection{Regularity of the Lyapunov exponent}

Let $\omega\in \mathbb{R}^d$ be fixed. For a matrix-valued function $G$, denote by $L(G)$ the maximal Lyapunov exponent of the cocycle $(\omega,G)$.

\begin{Coro}
Let $(\omega,A+F)$ satisfy the assumptions of Theorem~\ref{th: main theorem}. Then for all $A'\in U_{\Lambda}(\mathbb{T}^d,sl(2,\mathbb{R})$, 

\begin{equation}\label{lyapunov}|L(A')-L(A+F)| \leq 4 | A'-(A+F)|_0
\end{equation}

\end{Coro}

\begin{proof}
Let $Y$ be as in Theorem~\ref{th: main theorem} and $A_\infty$ the elliptic matrix such that 
\[ \partial_\omega Y = (A+F)Y-YA_\infty.\]
Let $A'\in U_{\Lambda, r}(\mathbb{T}^d,M_2(\mathbb{R}))$, then 
\[ \partial_\omega Y= A' Y - Y (A_\infty + Y^{-1} ( A'- (A+F))Y).\]
Notice that the Lyapunov exponent of $(\omega, A+F)$ is the same as for $(\omega, A_\infty)$, thus it is zero. Denoting by $L$ the Lyapunov exponent, one has
\begin{eqnarray*}
|L(A')-L(A+F)| &= & L(A')  = L( A_\infty + Y^{-1} ( A'- (A+F))Y)\\
&\leq &  \ln | A_\infty + Y^{-1} ( A'- (A+F))Y |_0\\
& \leq & \ln (1+|Y^{-1} ( A'- (A+F))Y |_0)\\
&\leq & | Y^{-1} ( A'- (A+F))Y|_0
\end{eqnarray*}
\noindent The bound on $Y$ gives the estimate \eqref{lyapunov}. $\Box$
\end{proof}

Let us remark that in order to describe the regularity without any condition on the fibered rotation number, a statement about almost reducibility would be needed.

\section{Preliminary reductions and choice of the sequences of parameters}\label{sequences}

Let us start with some preliminary lemmas. Recall that the spectrum of a non-zero elliptic matrix $A\in sl(2,\mathbb{R})$ is of the form $\mathrm{Spec}(A)=\{\pm i \alpha\}$ for some real number $\alpha>0$ (such a number is well-defined up to a sign); this real number will be called the rotation number of $A$ and denoted by $\rho(A)$. The following lemma gives us a real normal form for such an elliptic matrix and an estimate on the size of the transformation to normal form.

\begin{lemma}\label{lem:eliasson}
Given an elliptic matrix $A\in sl(2,\mathbb{R})$ with $\rho(A)=\alpha>0$, there exist a matrix $P \in SL(2,\mathbb{R})$ such that 
\[PAP^{-1}=\alpha J, \quad J=\begin{pmatrix}
0 & 1 \\
-1 & 0
\end{pmatrix}\]
and
\[ |P| \leq 2(|A|/\alpha)^{1/2}, \quad |P^{-1}| \leq 1. \]
\end{lemma}

For a proof, we refer to~\cite{hou2012almost}. It will sometimes be useful to use a complex normal form in which the matrix is diagonal. To do this, we consider the matrix
\begin{equation}\label{M}
M=\frac{1}{1=i}\begin{pmatrix}
1 & -i \\
1 & i
\end{pmatrix} \in U(2) 
\end{equation}
and we define the complex invertible matrix $Q=MP$ where $P$ is the matrix given by Lemma~\ref{lem:eliasson}, so that
\[QAQ^{-1}=i\alpha R, \quad R=\begin{pmatrix}
1 & 0 \\
0 & -1
\end{pmatrix}\]
and
\[ |Q| \leq 2(|A|/\alpha)^{1/2}, \quad |P^{-1}| \leq 1. \]
This lemma allows us to reduce the proof of Theorem~\ref{th: main theorem} to the case where $A=\alpha J$; indeed, it suffices to replace the smallness assumption on $|F|_r$ by a smallness assumption on $|PFP^{-1}|_r$ which is bounded by $2(|A|/\alpha)^{1/2} |F_r|$. In such a normal form, we have the following elementary lemma.

\begin{lemma}\label{lemma:rhoRegualirity}
Assume that $A=\alpha J$. Then
\[ |\rho(A)- \rho(A + F)|\leq4 |F|_0.\] 
\end{lemma}

For a proof, we refer to  Appendix~\ref{app2}. Finally, a matrix $A=\alpha J$ will remains elliptic under a small perturbation by a matrix in $sl(2,\mathbb{R})$ but no longer in normal from. Yet Lemma~\ref{lem:eliasson} immediately implies the following.

\begin{lemma}\label{constant}
Assume that $A=\alpha J$ with $\alpha>0$ and $B \in sl(2,\mathbb{R})$ such that $|B| \leq \varepsilon$. If $\varepsilon \leq \alpha/4$, then the matrix $A+B$ is elliptic with $\rho(A+B)=\beta$ satisfying
\[ \alpha/2  \leq \beta \leq \alpha+\alpha/2\]
and thus there exists $P\in SL(2,\mathbb{R}))$ such that $P(A+B)P^{-1}=\beta J$ with
\[|P| \leq 4 \quad |P^{-1}| \leq 1. \]
\end{lemma}

Upon these preliminary reductions, we can define a sequence of parameters for the iterations. Let $\rho(A)=\alpha$, $\rho(A+F)=\rho$, $\varepsilon=|F|_r$ and define for $\nu \in \mathbb{N}$
\[ \varepsilon_\nu=4^{-\nu}\varepsilon. \]
We now choose $N_0\in\mathbb{N}$, depending on $r$ and the approximating function $\Psi$, sufficiently large so that
\begin{equation}\label{N0}
\int_{N_0}^{+\infty} \frac{\Lambda'(v)\ln\Psi(v)dv}{\Lambda^2(v)} \leq \frac{\pi}{6} r.
\end{equation}
Then we set
\[ N_\nu = \Psi^{-1}(2^\nu\Psi(N_0)) \]
and observe that for all $\nu \in \mathbb{N}$, we have
\begin{equation}\label{condition:N and epsilon}
\Psi(N_\nu)\varepsilon_\nu=2^{-\nu}\Psi(N_0)^{}\varepsilon_0=2^{-\nu}\Psi(N_0)^{}\varepsilon.
\end{equation}
and thus the sequence $\Psi(N_\nu)\varepsilon_\nu$ is summable. We can define our threshold $\varepsilon \leq \bar{\varepsilon}$ by the requirements that
\[ \varepsilon \leq \alpha/4, \quad 2^8\Psi(N_0)\varepsilon \leq 1. \]
One then easily check that
\begin{equation}\label{bout1}
2^8 \Psi(N_\nu)\varepsilon_\nu \leq 1 
\end{equation}
holds true for all $\nu \in \mathbb{N}$. Next we define another sequence $\sigma_\nu>0$ for $\nu \in \mathbb{N}$ by 
\[ \sigma_{\nu}=\frac{3\ln 2}{\pi \Lambda(N_\nu)} \]
so that for all $\nu \in \mathbb{N}$, we have the equality
\begin{equation}\label{bout2}
2^6 e^{-2\pi\Lambda(N_\nu)\sigma_\nu}=1.
\end{equation}
Finally we define recursively the sequence $r_\nu$ by setting $r_0=r$ and $r_{\nu+1}=r_\nu-\sigma_\nu$; we will see later, as a consequence of~\eqref{N0} and of our choice of $\sigma_\nu$, that this sequence is well-defined (in the sense that $\sigma_\nu < r_\nu$), $r_\mu \geq r/2$ and thus $r_\nu$ converges to its infimum $r^*\geq r/2$.

\bigskip

\section{The iteration step}

Given $F : \mathbb{T}^d \rightarrow gl(2,\mathbb{R})$, we define  $\textrm{tr}\langle F\rangle=\textrm{tr}(\hat{F}(0))$ where $\hat{F}(0)$ is the average of $F$ with respect to Lebesgue measure. We also define truncation operators $T_N$ and $\dot{T}_N$ on $U_{\Lambda}(\mathbb{T}^d,gl(2,\mathbb{R}))$ as
\begin{eqnarray*}
(T_N F)(\theta) :=\sum_{|k|\leq N} \hat{F}(k)e^{i2\pi k\cdot \theta} ,\qquad
(\dot{T}_N F)(\theta) :=\sum_{0<|k|\leq N} \hat{F}(k)e^{i2\pi k\cdot\theta}.
\end{eqnarray*}

In this section, for a fixed $\nu \in\mathbb{N}$, we consider a cocycle $(\omega,A_\nu+F_\nu)$ which satisfy
\begin{equation}\label{rangnu}\tag{$H_\nu$}
\begin{cases}
A_\nu=\alpha_\nu J, \\
|F_\nu |_{r_\nu}\leq \varepsilon_\nu, \quad \varepsilon_\mu \leq \alpha_\nu/4, \quad \textrm{tr}\langle F_\nu\rangle=0, \\ 
\rho(A_\nu + F_\nu)=\rho.
\end{cases}
\end{equation}
We will conjugate this cocycle, by a transformation which is homotopic to the identity, to a cocycle $(\omega,A_{\nu+1}+F_{\nu+1})$ which satisfy~$H_{\nu+1}$ together with estimates on such a transformation. First we have the following obvious lemma.

\begin{lemma}
\label{lemma:rho(A)goodenough}
For any $k \in \mathbb{Z}^d$ such that $0<|k|\leq N_\nu$, we have
\begin{eqnarray*}
&&|2\alpha_\nu \pm2\pi k\cdot \omega |>
 \frac{1}{2\Psi({N_\nu})}.
\end{eqnarray*}
\end{lemma}

\begin{proof}
Indeed we have $\rho(A_\nu)=\alpha_\nu$ and $\rho(A_\nu + F_\nu)=\rho$ hence
\begin{eqnarray*}
|2\alpha_\nu \pm 2\pi k\cdot \omega |
&\geq&
|2\rho \mp 2\pi k\cdot \omega | - |2\rho(A_\nu+F_\nu) -2\rho(A_\nu)|\nonumber\\
&\geq& \frac{1}{\Psi(N_\nu)}-8\varepsilon_\nu \geq
 \frac{1}{2\Psi({N_\nu})}
\end{eqnarray*}
where we used Lemma \ref{lemma:rhoRegualirity} and the fact that $16\Psi(N_\nu)\varepsilon_\nu \leq 1$.
\end{proof}

Next we solve an approximate cohomological equation.

\begin{lemma}
\label{lemma:linearizedEquation}
If $G_\nu = \dot{T}_{N_\nu}F_\nu$, there is a unique $X_\nu$ such that $X_\nu = \dot{T}_{N_\nu}X_\nu$ satisfying the equation
\begin{equation}\label{eq:linearizedEquation}
\partial_\omega X_\nu  = [A_\nu,X_\nu] +G_\nu
\end{equation}
with the estimates
\[ |X_\nu|_{r_\nu} \leq \Psi(N_\nu)\varepsilon_\nu, \quad |(I+X_\nu)^{-1}|_{r_\nu} \leq 2. \]
Moreover $\textrm{tr}\langle X_\nu\rangle = 0$.

\end{lemma}
\begin{proof}
Observe that conjugating the cocycle by the complex matrix $M$ defined in~\eqref{M}, it is sufficient to prove the statement for $A_\nu$ in complex normal form $i\alpha R$. Expanding $F_\nu$ and $X_\nu$ in Fourier series, the equation \eqref{eq:linearizedEquation} yields
\[
\sum_{0<|k|\leq N_\nu}\partial_\omega \hat{X}_\nu(k) e^{2\pi i k\cdot \theta} =
\sum_{0<|k|\leq N_\nu}[A_\nu,\hat{X}_\nu(k)]e^{2\pi ik\cdot \theta} + \sum_{0<|k|\leq N_\nu} \hat{F_\nu}(k) e^{2\pi ik\cdot \theta},
\]
which is equivalent to
\[
\partial_{\omega} \hat{X_\nu}(k) = [A_\nu,\hat{X}_\nu(k)] +\hat{F_\nu}(k), \quad0<|k|\leq N_\nu.
\]
Since $A_\nu$ is diagonal, the solution of the above equation is
\begin{equation*}\label{eq:Linearized Eq of Fourier expansion}
\hat{X_\nu}(k)=L_k^{-1}\hat{F_\nu}(k),
\end{equation*}
where $L_k$ is the operator defined by
\[
L_k:sl(2,\mathbb{R})\to sl(2,\mathbb{R}),\quad \tilde{X}\mapsto2\pi ik\cdot \omega  \tilde{X}-[\tilde{A}_\nu,\tilde{X}].
\]
The spectrum of $L_k$ is $\{2\pi i  k\cdot \omega  \pm2\alpha_\nu,~2\pi ik\cdot \omega \}$.
By Lemma \ref{lemma:rho(A)goodenough} and~\eqref{eq:BR frequency omega}, the operator $L_k$ for $0<|k|\leq N_\nu$ is invertible with norm bounded by $2\Psi(N_\nu)$ so
\[
|\hat{X_\nu}(k)|=|L_k^{-1}\hat{F_\nu}(k)|\leq     2\Psi(N_\nu)|\hat{F_\nu}(k)|
\]
and thus
\[
|X_\nu|_{r_\nu}  =  \sum_{0< |k|\leq N_\nu} |\hat{X_\nu}(k)|e^{2\pi  \Lambda(k)r} \leq 2\Psi(N_\nu) \sum_{0< |k|\leq N_\nu} |\hat{F_\nu}(k)|e^{2\pi  \Lambda(k)r} \leq 2\Psi(N_\nu)\varepsilon_\nu.
\]
Since $4\Psi(N_\nu)\varepsilon_\nu \leq 1$ the estimate
\[
|(I+X_\nu)^{-1}|_{r_\nu}\leq \frac{1}{1-|X_\nu|_{r_\nu}} \leq 2 \]
is obvious and so is $\textrm{tr}\langle X_\nu\rangle  =0$ because $\hat{X}_\nu(0)=0$. This completes the proof.
\end{proof}

We can finally state our main iterative proposition.

\begin{proposition}
\label{lemma:fullEquation}
Let $(\omega,A_\nu+F_\nu)$ be as in~\eqref{rangnu}. Then there exists a transformation $Y_\nu$ homotopic to the identity such that
\[
\partial_\omega Y_\nu  = (A_\nu+F_\nu)Y_\nu -Y_\nu(A_{\nu+1} +F_{\nu+1}),
\]
with $(\omega,A_{\nu+1} +F_{\nu+1})$ satisfying $(H_{\nu+1})$ and such that
\[ |Y_\nu-I|_{r_\nu} \leq 8\Psi(N_\nu)\varepsilon_\nu \]
\end{proposition}

\begin{proof} 
The transformation $Y_\nu$ will be the composition of a quasi-periodic linear transformation given by Lemma~\ref{lemma:linearizedEquation} and a constant transformation given by Lemma~\ref{constant} to put back the constant elliptic part into normal form.

Let $X_\nu$ be given by Lemma \ref{lemma:linearizedEquation}, and let $Z_\nu=I+X_\nu$. Since $X_\nu$ solves
\begin{equation*}
\partial_\omega X_\nu  = [A_\nu,X_\nu] + G_\nu,
\end{equation*}
a computation leads to
\[
\partial_\omega Z_\nu  = (A_\nu+F_\nu)Z_\nu -Z_\nu(B_{\nu} +R_{\nu}),
\]
with
\[
B_\nu=A_\nu+\hat{F}_\nu(0), \quad R_{\nu} = (I+X_\nu)^{-1}[(F_\nu-\hat{F}_\nu(0)-G_\nu) + F_\nu X_\nu  - X_\nu\hat{F}_\nu(0).
\]
We can estimate
\begin{eqnarray*}
|R_{\nu}|_{r_\nu-\sigma_\nu} &=& |(I+X_\nu)^{-1}[(F_\nu-\hat{F}_\nu(0)-G_\nu) + F_\nu X_\nu  - X_\nu\hat{F}_\nu(0)|_{r_\nu-\sigma_\nu}\\
&\leq&  2|F_\nu-\hat{F}_\nu(0)-G_\nu|_{r_\nu-\sigma_\nu} + 2|F_\nu X_\nu|_{r_\nu-\sigma_\nu}  + 2|X_\nu\hat{F}_\nu(0)|_{r_\nu-\sigma_\nu}\\
&\leq&  |F_\nu|_{r_\nu}2e^{-2\pi\Lambda(N_\nu)\sigma_\nu} +4  |X_\nu|_r |F_\nu|_r\\
&\leq&  2e^{-2\pi\Lambda(N_\nu)\sigma_\nu}\varepsilon_\nu +8\Psi(N_\nu)\varepsilon_\nu \varepsilon_\nu\\
&\leq & 2^{-5}\varepsilon_\nu+2^{-5}\varepsilon_\nu
\end{eqnarray*}
where we used~\eqref{bout1} and~\eqref{bout2} in the last inequality, and therefore
\begin{equation}\label{est1}
|R_{\nu}|_{r_\nu-\sigma_\nu}  \leq 2^{-4}\varepsilon_\nu. 
\end{equation}
Let us now check that $\textrm{tr}\langle R_\nu\rangle=0$. By the equality
\[
\partial_\omega (I+X_\nu) = (A_\nu+F_\nu)(I+X_\nu)-(I+X_\nu)(B_{\nu}+R_{\nu}),
\]
we know that
\[
R_{\nu}= -(I+X_\nu)^{-1}\partial_\omega (I+X_\nu) + (I+X_\nu)^{-1}(A_\nu+F_\nu)(I+X_\nu)-B_{\nu}.
\]
It follows from the assumptions $\textrm{tr} A_\nu = \textrm{tr} \hat{F}_\nu(0) = \textrm{tr} B_{\nu} = 0$ that
\[
  \textrm{tr}\langle (I+X_\nu)^{-1}(A_\nu+F_\nu)(I+X_\nu)-B_{\nu}\rangle =0.
\]
On the other hand, using $\textrm{tr}(AB)=\textrm{tr}(BA)$ we have
\begin{eqnarray*}
 && \mathrm{tr} \langle  (I+X_\nu)^{-1}\partial_\omega (I+X_\nu) \rangle \\
 &=& \mathrm{tr}  \langle \partial_\omega (I+X_\nu)  (I+X_\nu)^{-1} \rangle\\
 &=& \mathrm{tr} \langle  (\partial_\omega X_\nu)  (I+X_\nu)^{-1} \rangle \\
 &=&  \mathrm{tr}\langle  \partial_\omega X_\nu - \frac{1}{2}((\partial_\omega X_\nu) X_\nu + X_\nu \partial_\omega X_\nu \rangle) +
 \frac{1}{3}((\partial_\omega X_\nu)X_\nu^2 + X_\nu (\partial_\omega X_\nu) X_\nu + (\partial_\omega X_\nu) X_\nu^2) +\cdots
 \\
 &=& \textrm{tr}\langle\sum_{k=0}\frac{(-1)^{k}}{k+1}\langle \partial_\omega(X_\nu^{k+1})\rangle \\
 &=&0.
\end{eqnarray*}
Thus $\textrm {tr}\langle R_{\nu}\rangle = 0$. Now we want to apply Lemma~\ref{constant}. Observe that $\varepsilon_\nu \leq \alpha_\nu/4$, therefore Lemma~\ref{constant} gives that $\rho(B_\nu)=\rho(A_\nu+\hat{F}_\nu(0))=\alpha_{\nu+1}$ satisfy
\[ \alpha_\nu/2 \leq \alpha_{\nu+1} \leq \alpha_\nu +\alpha_\nu/2  \]
hence there exists $P_\nu\in SL(2,\mathbb{R})$ such that $P_\nu(B_\nu)P_\nu^{-1}=\alpha_{\nu+1} J$ with
\begin{equation}\label{est2}
|P_\nu| \leq 4 \quad |P_\nu^{-1}| \leq 1.
\end{equation}
We can finally define
\[ Y_\nu=P_\nu Z_\nu P_\nu^{-1}, \quad A_{\nu+1}=\alpha_{\nu+1} J, \quad F_{\nu+1}=P_\nu R_\nu P_\nu^{-1} \]
so that
\[
\partial_\omega Y_\nu  = (A_\nu+F_\nu)Y_\nu -Y_\nu(A_{\nu+1} +F_{\nu+1}).
\]
Let us check that $(H)_{\nu+1}$ is satisfied. By definition $A_{\nu+1}=\alpha_{\nu+1} J$ and we know that
\[  \varepsilon_{\nu+1}=\varepsilon_\nu/4 < \alpha_\nu/8 \leq \alpha_{\nu+1}/4.  \] 
Again by definition,  $F_{\nu+1}=P_\nu R_\nu P_\nu^{-1}$ and from the estimates~\eqref{est1} and~\eqref{est2} we have
\[ |F_{\nu+1}|_{r_\nu-\sigma_\nu} \leq 2^{-4}4\varepsilon_\nu=\varepsilon_{\nu+1} \]
whereas $\textrm {tr}\langle F_{\nu+1}\rangle =\textrm {tr}\langle R_{\nu}\rangle = 0$. Then, $Z_\nu=I+X_\nu$ is obviously homotopic to the identity ans so is $P_\nu\in SL(2,\mathbb{R})$ hence
\[ \rho(A_{\nu+1} + F_{\nu+1})=\rho(A_\nu + F_\nu)=\rho.\]
To conclude, from~\eqref{est2} and the estimates on $X_\nu$ given by Lemma~\ref{lemma:linearizedEquation} we get
\[ |Y_\nu-I|_{r_\nu} \leq 4|Z_\nu-I|_{r_\nu} \leq 4|X_\nu|_{r_\nu} \leq 8\Psi(N_\nu)\varepsilon_\nu.  \]
This concludes the proof.
\end{proof}

\section{Proof of Theorem~\ref{th: main theorem}}\label{Proof}

In this section we finally prove Theorem~\ref{th: main theorem}. Letting $A_0=A$ and $F_0=F$, we can apply inductively Proposition~\ref{lemma:fullEquation} and for any $\nu \in \mathbb{N}$, if we define
\[ Y^\nu=\prod_{\mu\leq\nu}Y_\mu \in U_{\Lambda,r_\nu}(\mathbb{T}^d,GL(2,\mathbb{R})) \]
we have
\[
\partial_\omega Y^\nu  = (A+F)Y^\nu -Y^\nu(A_{\nu+1} +F_{\nu+1}),
\]
with
\[ \rho(A_{\nu+1} +F_{\nu+1})=\rho, \quad |F|_{r_{\nu+1}}\leq \varepsilon_{\nu+1}. \]
As the sequence $\varepsilon_\nu$ converges to zero, the only thing that remains to be proved is that $r_{\nu}$ converges to a non-zero limit and that $Y^\nu$ converges. To do so, let us first observe that
\[ \sum_{\nu \geq 1}\frac{1}{\Lambda(N_\nu)}=\sum_{\nu \geq 1}\frac{1}{\Lambda(\Psi^{-1}(2^\nu\Psi(N_0)))} \leq \int_{0}^{+\infty}\frac{dx}{\Lambda(\Psi^{-1}(2^x\Psi(N_0)))}  \]
and changing variables $v=\Psi^{-1}(2^x\Psi(N_0))$, this gives
\[ \sum_{\nu \geq 1}\frac{1}{\Lambda(N_\nu)} \leq \frac{1}{\ln 2}\int_{N_0}^{\infty} \frac{\Psi'(v)dv}{\Lambda(v)\Psi(v)} \]
and finally, by an integration by parts this yields
\[ \sum_{\nu \geq 0}\frac{1}{\Lambda(N_\nu)} \leq \frac{1}{\ln 2}\int_{N_0}^{\infty} \frac{\Lambda'(v)\ln \Psi(v)dv}{\Lambda^2(v)} \leq \frac{\pi}{6\ln 2}r \]
where the last inequality follows from~\eqref{N0}. By definition of $\sigma_\mu$ this gives
\[ \sum_{\nu \geq 0}\sigma_\nu =\frac{3 \ln 2}{\pi} \sum_{\nu \geq 0}\frac{1}{\Lambda(N_\nu)} \leq r/2\]
and thus $r_\nu \geq r/2$ converges to some $r^* \geq r/2$. To conclude, $Y^\nu-I$ is easily seen to form a Cauchy sequence on the space $U_{\Lambda,r/2}(\mathbb{T}^d,GL(2,\mathbb{R}))$ and thus $Y^\nu$ converges to a limit $Y$ which satisfies the bound
\[ |Y-I|_{r/2} \leq 2\sum_{\nu \geq 0}|Y_\nu-I| \leq 16\sum_{\nu \geq 0}\Psi(N_\nu)\varepsilon_\nu \leq 32\Psi(N_0)\varepsilon.  \]

\section{Proof of Theorem~\ref{th2}}\label{app3}

Let us start with the following lemma, which says that if $\omega$ does not satisfy the $\Lambda$-R\"{u}ssmann condition~\eqref{Rus}, then one cannot solve the cohomological equation in general.

\begin{lemma}\label{coho}
Assume $\omega$ does not satisfy~\eqref{Rus}, that is
\[ \limsup_{v\to\infty}\frac{\log\Psi(v)}{\Lambda(v)}>0. \]
Then there exist $r>0$ such that for any $\varepsilon\geq 0$ and any $\rho \in \mathbb{R}$, there exists a function $u : \mathbb{T}^d \rightarrow \mathbb{R}$ for which 
\[ \int_{\mathbb{T}^d}u(\theta)d\theta =\rho, \quad |u-\rho|_r \leq \varepsilon  \]
but such that the equation
\begin{equation}\label{cohomo}
\omega\cdot \partial_\theta v(\theta)=u(\theta)-\rho \tag{$E$}
\end{equation}
has no continuous solution $v : \mathbb{T}^d \rightarrow \mathbb{R}$.
\end{lemma}

\begin{proof}
Let $r>0$ be such that 
\[ \limsup_{v\to\infty}\frac{\log\Psi(v)}{\Lambda(v)}\geq 3\pi r \]
so that there exists a positive sequence $v_j \rightarrow +\infty$ for which
\[ \Psi(v_j)^{-1} \leq e^{-3\pi r\Lambda(v_j)}\]
By definition of $\Psi$, there exists infinitely many $k_j \in \mathbb{Z}^d \setminus \{0\}$ (for which $|k_j|=v_j$) and
\[ |2\pi k_j \cdot \omega| \leq e^{-3\pi r\Lambda(|k_j|)}. \]
Let us define a constant 
\[ C=\sum_{j \in \mathbb{N}} e^{-\pi r\Lambda(|k_j|)} <+\infty \]
and a function 
\[ u(\theta)=\sum_{k\in\mathbb{Z}^d} \hat{u}(k) e^{2\pi i k\cdot \theta},\]
by setting
\[ \hat{u}(0)=\rho, \quad \hat{u}(k_j)=\varepsilon C^{-1}2\pi k_j \cdot \omega \]
and all other Fourier coefficients equal to zero. By construction we do have
\[ \int_{\mathbb{T}^d}u(\theta)d\theta =\rho, \quad |u-\rho|_r \leq \varepsilon. \]
Now a function $v : \mathbb{T}^d \rightarrow \mathbb{R}$
\[ v(\theta)=\sum_{k\in\mathbb{Z}^d} \hat{v}(k) e^{2\pi i k\cdot \theta}\]
solves~\eqref{cohomo} if and only if
\[ \hat{v}(k_j)=\frac{\hat{u}(k_j)}{2\pi ik_j\cdot \omega}=\varepsilon C^{-1}. \]
Clearly such Fourier coefficients do not define a function which is integrable, and therefore $v$ cannot be continuous.
\end{proof}

To conclude the proof of Theorem~\ref{th2}, let $u : \mathbb{T}^d \rightarrow \mathbb{R}$ be the function given by Lemma~\ref{coho}, and consider the $sl(2,\mathbb{R})$ cocycle $(\omega,A+F)$ defined by
\[ A=\rho J, \quad F(\theta)=u(\theta)J-\rho J. \]
Its fibered rotation number is equal to $\rho$ (see Appendix~\ref{app2}) and $|F|\leq \varepsilon$. Argue by contradiction that $(\omega,A+F)$ is reducible. Since it takes values in $so(2,\mathbb{R})$, it follows from~\cite{chavaudret2011reducibility} that it is reducible by a transformation that takes values in $SO(2,\mathbb{R})$, therefore there exists $v : \mathbb{T}^d \rightarrow \mathbb{R}$ and
\[  Y(\theta)=\begin{pmatrix}
\cos v(\theta) & \sin v(\theta) \\
-\sin v(\theta) & \cos v(\theta)
\end{pmatrix} \in SO(2,\mathbb{R}) \]
such that 
\[\partial_\omega Y=MY-YB\] 
for some constant matrix $B=\beta J$. But then necessarily $\beta=\rho$ (that is $B=A$) and the above equation is equivalent to~\eqref{cohomo} which has no continuous solution, which is a contradiction.

\section{Appendix}

\subsection{Product estimates}\label{app1}

\begin{lemma} \label{lem: appendix banach algebra}
Suppose $\Lambda$ satisfy the subadditivity condition~\eqref{S}. For any $f,g\in U_{\Lambda,r}(\mathbb{T}^d,\mathbb{R})$, we have $fg\in U_{\Lambda,r}(\mathbb{T}^d,\mathbb{R})$  and
\[
|f g|_r\leq |f|_r|g|_r.
\]
\end{lemma}
\begin{proof}

Expanding in Fourier series we have
\begin{eqnarray*}
f(\theta)=\sum_{k\in\mathbb{Z}^d} \hat{f}(k) e^{2\pi i k\cdot \theta}, \quad  g(\theta) = \sum_{k\in\mathbb{Z}^d} \hat{g}(k) e^{2\pi i k\cdot \theta}
\end{eqnarray*}
and
\begin{eqnarray*}
f(\theta) g(\theta) &=& \Big(\sum_{m\in\mathbb{Z}^d} \hat{f}(m) e^{2\pi im \cdot \theta} \Big) \Big(\sum_{n\in\mathbb{Z}^d} \hat{g}(n) e^{2\pi i n\cdot \theta}\Big)\\
&=&  \sum_{k\in\mathbb{Z}^d} \Big(\sum_{m+n=k}\hat{f}(m) \hat{g}(n)\Big) e^{2\pi ik\cdot \theta}.
\end{eqnarray*}
On the one had
\begin{eqnarray*}
|f g|_r &=&  |\sum_{k\in\mathbb{Z}^d} \Big(\sum_{m+n=k}\hat{f}(m) \hat{g}(n)\Big) e^{2\pi ik\cdot \theta}|_r\\
&=&\sum_{k\in\mathbb{Z}^d} \Big|\sum_{m+n=k}\hat{f}(m) \hat{g}(n)\Big| e^{2\pi  \Lambda(|k|)r}\\
&\leq&\sum_{k\in\mathbb{Z}^d} \Big(\sum_{m+n=k}|\hat{f}(m) \hat{g}(n)| e^{2\pi  \Lambda(|k|)r}\Big)
\end{eqnarray*}
and on the other hand
\begin{eqnarray*}
|f|_r|g|_r &=& |\sum_{m\in\mathbb{Z}^d} \hat{f}(m) e^{2\pi i m\cdot \theta}|_r |\sum_{n\in\mathbb{Z}^d} \hat{g}(n) e^{2\pi i n\cdot \theta}|_r\\
&=& \Big(\sum_{m\in\mathbb{Z}^d} |\hat{f}(m)| e^{2\pi \Lambda(|m|)r} \Big) \Big(\sum_{m\in\mathbb{Z}^d} |\hat{g}(n)| e^{2\pi \Lambda(|n|)r} \Big)\\
&= & \sum_{k\in\mathbb{Z}^d} \sum_{m+n=k}|\hat{f}(m) \hat{g}(n)| e^{2\pi r(\Lambda(|m|)+\Lambda(|n|))} .
\end{eqnarray*}
Since $|k| =|m+n|\leq|m|+|n|$ we have $\Lambda(|k|)\leq \Lambda(|m|)+\Lambda(|n|)$ by subadditivity and thus
$|f\cdot g|_r\leq |f|_r|g|_r$.
\end{proof}

\subsection{Fibered rotation number}\label{app2}

Let us consider a quasi-periodic cocycle
\[  x'(t) =M(t\omega)x(t),\]
where
\[ M : \mathbb{T}^d \rightarrow sl(2,\mathbb{R}), \quad M(\theta)=\begin{pmatrix}
a(\theta) & b(\theta) \\
c(\theta) & -a(\theta).
\end{pmatrix}  \] 
with $\omega \in \mathbb{R}^d$ non-resonant. Following Eliasson~\cite{eliasson1992floquet}, we define the fibered rotation number $\rho(M) \in \mathbb{R}$ as
\[ \rho(M)=\lim_{t \rightarrow + \infty} \frac{\phi_M((t)}{t}\]
where $\phi_M(t)=\phi(t)$ is any solution of the equation
\[ \phi'(t)=2a(t\omega)\cos\phi(t)\sin\phi(t)-(b(t\omega)+c(t\omega))\cos^2\phi(t)+b(t\omega). \]
A first special case is
\[ M(\theta)=u(\theta)J, \quad u : \mathbb{T}^d \rightarrow \mathbb{R} \]
we then have
\[  \rho(M)=\lim_{t \rightarrow + \infty} \frac{\phi_M((t)}{t}=\lim_{t \rightarrow + \infty} \frac{1}{t}\int_0^t\phi_M'(s)ds=\lim_{t \rightarrow + \infty} \frac{1}{t}\int_0^tu(s\omega)ds=\int_{\mathbb{T}^d} u(\theta)d\theta\]
where the last equality follows from the unique ergodicity of the translation flow by $\omega$. Another special case is
\[ M(\theta)=\alpha J+F(\theta), \quad \sup_{\theta \in \mathbb{T}^d}|F(\theta)|=\sup_{\theta \in \mathbb{T}^d}\left|\begin{pmatrix}
f_1(\theta) & f_2(\theta) \\
f_3(\theta) & -f_1(\theta)
\end{pmatrix}\right| \leq \varepsilon \]
from which one immediately obtains the estimate
\begin{eqnarray*}
|\rho(\alpha J)-\rho(M)| & = & |\alpha-\rho(M)| \leq  \left|\frac{1}{t}\int_0^t \phi_F'(s) ds\right| \leq 4\varepsilon
\end{eqnarray*}
which is nothing but Lemma~\ref{lemma:rhoRegualirity}.

\bibliographystyle{elsarticle-num-names}

\end{document}